\documentclass[12pt]{article}
\usepackage{amsmath}
\usepackage{amssymb}
\usepackage{amsthm}
\usepackage{mathtools} 
\mathtoolsset{centercolon}
\usepackage{url}

\DeclareUrlCommand\email{\urlstyle{rm}}

\numberwithin{equation}{section}

\theoremstyle{plain}
\newtheorem{theorem}{Theorem}[section]
\newtheorem{lemma}[theorem]{Lemma}
\newtheorem*{HN}{The Hilbert Nullstellensatz}

\theoremstyle{definition}
\newtheorem{example}[theorem]{Example}

\theoremstyle{remark}
\newtheorem{remark}[theorem]{Remark}

\title{On Polynomial Interpolation on Arbitrary Varieties}
\author{Tom McKinley, Boris Shekhtman, Brian Tuesink \\
Department of Mathematics and Statistics \\
University of South Florida, Tampa, Florida 33620, USA \\
E-Mail: \email{shekhtma@usf.edu} \\
}

\begin{document}
\maketitle

\begin{abstract}
To the best of our knowledge this paper is the first attempt to introduce and study polynomial interpolation of the polynomial data given on arbitrary varieties. In the first part of the paper we present results on the solvability of such problems. In the second part of the paper we relate the interpolation problem to polynomial solution of some boundary values problems. In particular, we extend a result of W. K. Hayman and Z. G. Shanidze \cite{Hay:Sha}.
\end{abstract}

\section{Introduction}\label{sec1}

This paper is a first attempt to introduce and study the problem of polynomial interpolation on curves, or more generally, on varieties in $\mathbb{C}^d$. The interpolation problem had mostly been studied on points. The singular exception is the study of interpolation on ``flats'' done by Carl de Boor, Nira Dyn and Amos Ron in \cite{deB:Dyn:Ron}. To the best of our knowledge this is the first effort to study interpolation on general varieties, at least in the field of approximation theory.

Here is the first general question: Let $V_1,\dotsc,V_n$ be subsets of $\mathbb{C}^d$ or $\mathbb{R}^d$ and let $p_1,\dotsc,p_n$ be polynomials in $\mathbb{C}\left[\mathbf{x}\right]:=\mathbb{C}\left[x_1,\dotsc,x_d\right]$ that we will refer to as data. When does there exist a polynomial $f\in\mathbb{C}
\left[x_1,\dotsc,x_d\right]$ such that the restriction $f$ onto $V_j$, denoted by $f\mid V_j$, coincides with $p_j$ on $V_j$ for all $j=1,\dotsc,n$? Notice that if such an interpolant $f$ exists then $V_j\subset\mathcal{V}_j:=\{\mathbf{x\in\mathbb{C}}^d:\left(f-p_j\right)(\mathbf{x})=0\}$ which is a variety in $\mathbb{C}^d$ and thus $f$ interpolates $p_j$ on the variety $\mathcal{V}_j$. For this reason, without loss of generality, we can restrict ourselves to interpolation on varieties $\mathcal{V}_1,\dotsc,\mathcal{V}_n$. Recall that a variety $\mathcal{V}\subset\mathcal{C}^d$ is defined as a set such that
$$
\mathcal{V}=\mathcal{V}(J):=\left\{x\in\mathbb{C}^d:p(\mathbf{x})=0\text{ for all }p\in J\right\}
$$
for some ideal $J\in\mathbb{C}(\mathbf{x})$.

It is clear that, in order for the interpolant $f$ to exist, the data have to coincide on the intersection of the varieties, i.e.,
\begin{equation}
\label{eq1.1}
p_j=p_k\text{ on }\mathcal{V}_j\cap\mathcal{V}_k.
\end{equation}%
However, as will be shown in a moment, this condition is not sufficient in general. Here is a quick example:

\begin{example}\label{exam1.1}
Let $\mathcal{V}_1=\left\{(x,y)\in\Bbbk^2:y=x^2\right\}$ and $\mathcal{V}_2=\{(x,y)\in\Bbbk:y=0\}=\{(x,0)\in\Bbbk^2\}$ suppose that $p_1(x,y)=x$ and $p_2(x,y)=0$. Clearly \eqref{eq1.1} holds. If $f(x,y)=p_2(x,y)=0$ on $\mathcal{V}_2$ then $f(x,0)=0$ and $f$ must have a factor of $y$, i.e., $f(x,y)=yf_1(x,y)$ for some polynomial $f_1$. But then $f\mid\mathcal{V}_1=yf_1(x,y)\mid\mathcal{V}_{1}=x^2f_1\left(x,x^2\right)\ne x$.
\end{example}

In this case, the intersection of the two varieties is a double-point, hence $p_1$ and $p_2$ would need to agree at that point not only in value but also in $D_x$ derivative.

The sufficiency depends on the varieties, or more precisely, on polynomial ideals generated by these varieties.

In Section \ref{sec2} of the paper we will give a necessary and sufficient condition on the varieties $\mathcal{V}_1,\dotsc,\mathcal{V}_n$ so that the interpolation problem with data satisfying \eqref{eq1.1} is always solvable. We will also introduce conditions on the data that allow the interpolation problem to be solvable for a general pair of varieties.

In Section \ref{sec3} we will seek solutions to the interpolation problem among polynomial solution to PDE's with constant coefficients, thus pointing out a connection of the subject to the very general boundary problems for PDE's. In particular we will extend some results of W. K. Hayman and Z. G. Shanidze of \cite{Hay:Sha} who considered the problem in two variables and for homogeneous PDE's of degree $2$.

The best news is that the proofs are extremely simple and rely on nothing more than the Hilbert Nullstellensatz.

We will now recall some rudimentary notations and results from algebraic geometry. All of them can be found in \cite{Cox:Lit:Osh}.

For a field $\Bbbk$ the algebra of all polynomials in $d$ variables with coefficients in $\Bbbk$ is denoted by $\Bbbk\left[\mathbf{x}\right]:=\Bbbk\left[x_1,\dotsc,x_d\right]$. Every polynomial $p\in\Bbbk\left[\mathbf{x}\right]$ can be written as
$$
p(\mathbf{x})=\sum_{\left\vert\alpha\right\vert\le n}a_a\mathbf{x}^{\alpha}
$$
for some $n$, where the monomials $\mathbf{x}^{\alpha}:=x_1^{\alpha_1}\dotsc x_d^{\alpha_d}$ for some nonnegative integers $\alpha_1,\dotsc,\alpha_d$ and $\left\vert\alpha\right\vert=\sum\limits_{j=1}^d\alpha_j$. If at least one of the coefficients $a_{\alpha}$ with $\left\vert\alpha\right\vert=n$ is different from zero we will say that the polynomial $p$ has degree $n$; we will use $p^{\wedge}$ to denote the leading form of such $p$:
$$
p^\wedge(\mathbf{x}):=\sum_{\left\vert\alpha\right\vert=n}a_{a}\mathbf{x}^{\alpha}.
$$
If $A$ is a set of polynomials in $\Bbbk\left[\mathbf{x}\right]$ we will use $A^{\wedge}$ to denote the set of all leading forms of polynomials in $A$.

For an ideal $J\subset\Bbbk\left[\mathbf{x}\right]$ we use $\mathcal{V}(J)$ to denote the associated variety:
$$
\mathcal{V}(J):=\{\mathbf{x}\in\Bbbk^{d}:f(\mathbf{x})=0\text{ for all }f\in J\}.
$$
A variety $\mathcal{V}$ in $\Bbbk^d$ is a subset of $\Bbbk^d$ such that $\mathcal{V=V}(J)$ for some ideal $J\subset\Bbbk\left[\mathbf{x}\right]$.

An ideal $J$ is called radical if $f^n\in J$ implies $f\in J$. In particular, with every variety $\mathcal{V}\subset\Bbbk^d$ we associate the radical ideal
$$
J(\mathcal{V}):=\left\{f\in\Bbbk\left[\mathbf{x}\right]:f(\mathbf{x})=0\text{ for all }\mathbf{x}\in\mathcal{V}\right\}.
$$
It is clear that for two varieties $\mathcal{V}_1$ and $\mathcal{V}_2$ the ideal $J\left(\mathcal{V}_1\cup\mathcal{V}_2\right)=J\left(\mathcal{V}_1\right)\cap J\left(\mathcal{V}_2\right)$. Also for the ideals $J_1$ and $J_2$ the sum $J_1+J_2$ is again an ideal and $\mathcal{V}\left(J_1+J_2\right)=\mathcal{V}\left(J_1\right)\cap\mathcal{V}\left(J_2\right)$.

We now come to the main ingredient of the proofs in the paper:

\begin{HN}
Let $J$ be an ideal in $\mathbb{C}\left[\mathbf{x}\right]$. If a polynomial $p\in\mathbb{C}\left[\mathbf{x}\right]$ vanishes on $\mathcal{V}(J)$ then there exists an integer $n$ such that $p^{n}\in J$.
\end{HN}

In particular, if $\mathcal{V}(J)=\varnothing$ then $1\in J$ and hence $J=\mathbb{C}\left[\mathbf{x}\right]$. If an ideal $J$ is radical and $p\in\mathbb{C}\left[\mathbf{x}\right]$ vanishes on $\mathcal{V}(J)$ then $p\in J$.

\section{Interpolation on general varieties}\label{sec2}

\subsection{Interpolation on disjoint varieties}\label{sec2.1}

In this subsection we will prove that interpolation on disjoint varieties is always possible. We will start with a simple lemma, which, peculiarly, is a form of ``Hahn--Banach Separation Theorem'' for varieties:

\begin{lemma}\label{lem2.1}
Let $\mathcal{V}_1$ and $\mathcal{V}_2$ be two disjoint varieties in $\mathbb{C}^d$. Then there exists a polynomial $f$ such that $f\mid\mathcal{V}_1=1$ and $f\mid\mathcal{V}_2=0$.
\end{lemma}

\begin{proof}
Since $\mathcal{V}_1\cap\mathcal{V}_2=\varnothing$ it follows from the Hilbert Nullstellensatz that $1\in J\left(\mathcal{V}_1\right)+J\left(\mathcal{V}_2\right)$, hence there are polynomials $g_1\in J\left(\mathcal{V}_1\right)$ and $g_2\in J\left(\mathcal{V}_2\right)$ such that
$$
1=g_1+g_2.
$$

Now $g_2$, as an element of $J\left(\mathcal{V}_2\right)$ vanishes on $\mathcal{V}_2$. Since $g_1$ vanishes on $\mathcal{V}_1$ it follows that $g_{2}\mid\mathcal{V}_1=1$ hence the polynomial $f=g_2$ satisfies the conclusion of the lemma.
\end{proof}

\begin{lemma}\label{lem2.2}
Given a collection of pairwise non-intersecting varieties $\mathcal{V}_1,\dotsc,\mathcal{V}_n$ in $\mathbb{C}^d$ then for every $k=1,\dotsc,n$ there exists a polynomial $f_{k}$ such that
$$
f_k\mid\mathcal{V}_k=1\text{ and }f_k\mid\mathcal{V}_j=0
$$
for all $j\ne k$.
\end{lemma}

\begin{proof}
Since $\bigcup\limits_{j\ne k}\mathcal{V}_j$ is a variety, the result follows from the previous lemma.
\end{proof}

\begin{theorem}\label{thm2.1}
Given a collection of pairwise non-intersecting varieties $\mathcal{V}_1,\dotsc,\mathcal{V}_n$ in $\mathbb{C}^d$ and arbitrary polynomials $p_1,\dotsc,p_n$ in $\mathbb{C}\left[\mathbf{x}\right]$ there exists a polynomial $f\in\mathbb{C}\left[\mathbf{x}\right]$ such that $f\mid\mathcal{V}_k=p_{k}\mid\mathcal{V}_k$ for all $k=1,\dotsc,n$.
\end{theorem}

\begin{proof}
Using polynomials $f_{k}$ from the previous lemma, the polynomial
$$
f=\sum_{k=1}^{n}p_{k}f_{k}
$$
gives us the desired interpolant.
\end{proof}

Notice that the polynomial $f$ constructed above is the analog for varieties of the classical Lagrange interpolation polynomial.

\subsection{Interpolation on general varieties}\label{sec2.2}

\begin{lemma}\label{lem2.4}
Let $\mathcal{V}_1$ and $\mathcal{V}_2$ be two varieties in $\mathbb{C}^d$. Then the following are equivalent:
\begin{enumerate}
\item[\rm a)]  For every pair of polynomials $p_1,p_2\in\mathbb{C}\left[\mathbf{x}\right]$ such that
\begin{equation}
p_1(\mathbf{x})=p_2(\mathbf{x})\text{ for }\mathbf{x}\in\mathcal{V}_1\cap\mathcal{V}_2
\label{eq2.1}
\end{equation}
there exists a polynomial $f\in\mathbb{C}\left[\mathbf{x}\right]$ such that
\begin{equation}
f\mid\mathcal{V}_k=p_k\mid\mathcal{V}_k\text{ for }k=1,2.
\label{eq2.2}
\end{equation}

\item[\rm b)]  The ideal $J\left(\mathcal{V}_1\right)+J\left(\mathcal{V}_2\right)$ is radical.
\end{enumerate}
\end{lemma}

\begin{proof}
Assume that the ideal $J\left(\mathcal{V}_1\right)+J\left(\mathcal{V}_2\right)$ is radical. Since $\mathcal{V}\left(J\left(\mathcal{V}_1\right)+J\left(\mathcal{V}_2\right)\right)=\mathcal{V}_1\cap\mathcal{V}_2$, it follows that $p_1-p_2$ vanishes on $\mathcal{V}\left(J\left(\mathcal{V}_1\right)+J(\mathcal{V}_2)\right)$ and by the Hilbert Nullstellensatz some power $\left(p_1-p_2\right)^n\in J\left(\mathcal{V}_1\right)+J\left(\mathcal{V}_2\right)$. But since $J\left(\mathcal{V}_1\right)+J\left(\mathcal{V}_2\right)$ is radical, it follows that there exist polynomials $g_1\in J\left(\mathcal{V}_1\right)$ and $g_2\in J\left(\mathcal{V}_2\right)$ such that $p_1-p_2=g_1+g_2$.

Now the polynomial $g_1\mid\mathcal{V}_1=0$ hence $g_2(\mathbf{x})=p_1(\mathbf{x})-p_2(\mathbf{x})$ for $\mathbf{x}\in\mathcal{V}_1$; on the other hand $g_2\mid\mathcal{V}_2=0$ hence
$$
f:=g_{2}+p_{2}
$$
satisfies the interpolation condition \eqref{eq2.2}.

Next, assume that $J\left(\mathcal{V}_1\right)+J\left(\mathcal{V}_2\right)$ is not radical. Then there exists a non-zero polynomial $g\in\mathbb{C}
\left[x_1,\dotsc,x_d\right]$ such that $g^{n}\in J\left(\mathcal{V}_1\right)+J\left(\mathcal{V}_2\right)$ for some $n>1$ but $g\notin J\left(\mathcal{V}_1\right)+J\left(\mathcal{V}_2\right)$. Choose $p_1=g$ and $p_2=0$.

Then, since $g^n\in J\left(\mathcal{V}_1\right)+J\left(\mathcal{V}_2\right)$, it follows that $g\mid\left(\mathcal{V}_1\cap\mathcal{V}_2\right)=0$ and \eqref{eq2.1} holds. Let $f\in\mathbb{C}\left[x_1,\dotsc,x_d\right]$ be such that \eqref{eq2.2} holds. Then $f\mid\mathcal{V}_2=0$ and, since $J\left(\mathcal{V}_2\right)$ is radical it follows that $f\in J\left(\mathcal{V}_2\right)$.

On the other hand $(g-f)\mid\left(\mathcal{V}_1\right)=0$ and since $J\left(\mathcal{V}_1\right)$ is radical we conclude that $(g-f)\in J\left(\mathcal{V}_1\right)$. But this implies that $g=(g-f)+f\in J\left(\mathcal{V}_1\right)+J\left(\mathcal{V}_2\right)$ which contradicts our assumption on $g$.
\end{proof}

This lemma explains Example \ref{exam1.1} from the Introduction. The ideals there are $J(\mathcal{V}_1)=\left\langle x^2-y\right\rangle$ and $J\left(\mathcal{V}_2\right)=\langle y\rangle$ are both radical, yet $J\left(\mathcal{V}_1\right)+J\left(\mathcal{V}_2\right)=\left\langle x^2-y,y\right\rangle=\left\langle x^2,y\right\rangle$ contains $x^2$ but not $x$, hence is not radical.

\begin{theorem}\label{thm2.5}
Let $\mathcal{V}_1,\dotsc,\mathcal{V}_n$ be a collection of varieties in $\mathbb{C}^d$. The following are equivalent:
\begin{enumerate}
\item[\rm a)]  For every collection of polynomials $p_1,\dotsc,p_n\in\mathbb{C}\left[\mathbf{x}\right]$ such that%
\begin{equation}
p_k(\mathbf{x})=p_j(\mathbf{x})\text{ for all }\mathbf{x}\in\mathcal{V}_k\cap\mathcal{V}_j\text{ for all }k,j=1,\dotsc,n
\label{eq2.3}
\end{equation}
there exists a polynomial $f$ such that $f\mid\mathcal{V}_k=p_k\mid\mathcal{V}_k$ for $k=1,\dotsc,n$.

\item[\rm b)]  For every $m<n$ the ideal $J\left(\bigcup\limits_{j=1}^m\mathcal{V}_j\right)+J\left(\mathcal{V}_m\right)$ is radical.
\end{enumerate}
\end{theorem}

\begin{proof}
Assume that b) holds. We prove a) by induction on $n$. The result is obviously true for $n=1$. Assuming that it is true for some $n$, let $\mathcal{V}_1,\dotsc,\mathcal{V}_n,\mathcal{V}_{n+1}$ be a collection of varieties in $\mathbb{C}^d$ such that for every $m<n+1$ the ideal $J\left(\bigcup\limits_{j=1}^m\mathcal{V}_j\right)+J\left(\mathcal{V}_m\right)$ is radical. Let $p_1,\dotsc,p_n,p_{n+1}\in\mathbb{C}\left[\mathbf{x}\right]$ be such that
$$
p_k\mid\left(\mathcal{V}_k\cap\mathcal{V}_j\right)
=p_j\mid\left(\mathcal{V}_k\cap\mathcal{V}_j\right)
\text{ for all }k,j=1,\dotsc,n+1
$$
then, by inductive assumption, there exists a polynomial $f_1$ such that $f_1\mid\mathcal{V}_k=p_k\mid\mathcal{V}_k$ for $k=1,\dotsc,n$. From \eqref{eq2.3} it follows that
\begin{equation}
f_1=p_{n+1}\text{ on }\left(\bigcup_{j=1}^n\mathcal{V}_{j}\right)\cap\mathcal{V}_{n+1}
\label{eq2.4}
\end{equation}
and, since $J\left(\bigcup\limits_{j=1}^n\mathcal{V}_j\right)+J\left(\mathcal{V}_{n+1}\right)$ is radical, Lemma \ref{lem2.4} implies that there exists a polynomial $f$ such that
$$
f=f_1\text{ on }\bigcup_{j=1}^n\mathcal{V}_j\text{ and }f\mid\mathcal{V}_{n+1}
=p_{n+1}\mid\mathcal{V}_{n+1}
$$
and, by \eqref{eq2.4}, the result follows.

The converse follows immediately from Lemma \ref{lem2.4}. Indeed if $J\left(\bigcup\limits_{j=1}^{m-1}\mathcal{V}_j\right)+J\left(\mathcal{V}_m\right)$ is not radical then we can choose a polynomial $f_1$ on $\bigcup\limits_{j=1}^{m-1}\mathcal{V}_j$ and $p_m$ on $\mathcal{V}_m$ such that no polynomial $f$ interpolates $f_1$ and $p_m$. Choosing $p_j=f_1$ for all $j=1,\dotsc m-1$ we obtain our result.
\end{proof}

A couple of remarks are in order:
\begin{enumerate}
\item[1)]  Unfortunately the condition b) in the previous theorem does not follow from ``the pairwise version'' of this condition: $J\left(\mathcal{V}_j\right)+J\left(\mathcal{V}_k\right)$ is radical for all $j,k=1,\dotsc,n$. Indeed, consider the following three varieties in $\mathbb{C}^2:\mathcal{V}_1=\mathcal{V}(\langle x\rangle)$, $\mathcal{V}_2=\mathcal{V}(\langle y\rangle)$ and $\mathcal{V}_3=\mathcal{V}(\langle x-y\rangle)$. Then $J\left(\mathcal{V}_1\cup
\mathcal{V}_2\right)=\mathcal{V}(\langle xy\rangle)$ and $J\left(\mathcal{V}_1\cup\mathcal{V}_2\right)+J\left(\mathcal{V}_3\right)=\langle xy,x-y\rangle$. This ideal contains $x^2=x(x-y)+xy$ yet does not contain $x$ and therefore is not a radical ideal.

\item[2)]  Notice that the statement a) in the theorem does not depend on ordering the varieties $\mathcal{V}_1,\dotsc,\mathcal{V}_n$ but the statement b), formally speaking, does. Hence the statement b) is equivalent to the following more general statement:
\begin{enumerate}
\item[c)]  For every proper subset $M\subset\{1,\dotsc,n\}$ and any $k\notin M$ the ideals $J\left(\cup_{m\in M}\mathcal{V}_m\right)+J\left(\mathcal{V}_k\right)$ is radical.
\end{enumerate}
\end{enumerate}

\subsection{Interpolation with restricted data}\label{sec2.3}

In this subsection we will address the case when the condition b) of Theorem \ref{thm2.5} does not hold. We will start with the lemma for two varieties.

\begin{lemma}\label{lem2.6}
Let $\mathcal{V}_1$ and $\mathcal{V}_2$ be two varieties in $\mathbb{C}^d$ and let $p_1,p_2$ be the data on these varieties. Then there exists a polynomial $f$ such that
\begin{equation}
f\mid\mathcal{V}_k=p_k\mid\mathcal{V}_k\text{ for }k=1,2,
\label{eq2.5}
\end{equation}
if and only if for every $\lambda\in\left[J(\mathcal{V}_1)+J\left(\mathcal{V}_2\right)\right]^{\bot}$ we have $\lambda(p_1)=\lambda(p_2)$.

Moreover, in this case any $f$ satisfying \eqref{eq2.5} also satisfies the condition
$$
\lambda(f)=\lambda\left(p_1\right)=\lambda\left(p_2\right)
$$
for every $\lambda\in\left[J\left(\mathcal{V}_1\right)+J\left(\mathcal{V}_2\right)\right]^{\bot}$.
\end{lemma}

\begin{proof}
If $\lambda\left(p_1\right)=\lambda\left(p_2\right)$ for all $\lambda\in\left[J\left(\mathcal{V}_1\right)+J\left(\mathcal{V}_2\right)\right]^{\bot}$ then $\lambda\left(p_1-p_2\right)=0$ and hence $p_1-p_2\in J\left(\mathcal{V}_1\right)+J\left(\mathcal{V}_2\right)$. Hence there exist $g_1\in J\left(\mathcal{V}_1\right)$ and $g_{2}\in J(\mathcal{V}_{2})$ such that
$$
p_1-p_2=g_1+g_2.
$$
Once again choosing $f:=p_2+g_2$ we deduce that $f=p_2+g_2=p_1-g_1$ interpolates $p_2$ on $\mathcal{V}_2$ since $g_2$ vanishes on $\mathcal{V}_2$ and interpolates $p_1$ on $\mathcal{V}_1$ since $g_1$ vanishes on $\mathcal{V}_1$.

For the ``moreover'' part, observe that $f$ constructed above satisfies $f-p_2=g_2\in J\left(\mathcal{V}_2\right)\subset J\left(\mathcal{V}_1\right)+J\left(\mathcal{V}_2\right)$, hence $\lambda\left(f-p_2\right)=0$ for all $\lambda\in\left[J\left(\mathcal{V}_1\right)+J\left(\mathcal{V}_2\right)\right]^{\bot}$. Similarly, $f-p_1=-g_1\in J\left(\mathcal{V}_1\right)\subset J\left(\mathcal{V}_1\right)+J\left(\mathcal{V}_2\right)$ and $f-p_2=g_2\in J\left(\mathcal{V}_2\right)\subset J\left(\mathcal{V}_1\right)+J\left(\mathcal{V}_2\right)$. Finally if $f_1$ also satisfies \eqref{eq2.5} then $\left(f-f_1\right)\mid\mathcal{V}_1=0$ and,
since $J\left(\mathcal{V}_1\right)$ is a radical ideal, $\left(f-f_1\right)\in J\left(\mathcal{V}_1\right)\subset J\left(\mathcal{V}_1\right)+J\left(\mathcal{V}_2\right)$.
\end{proof}

\begin{remark}
Returning to the example in the introduction, $\mathcal{V}_1=\left\{(x,y)\in\Bbbk^2:y=x^2\right\}$ and $\mathcal{V}_2=\{(x,y)\in\Bbbk^2:y=0\}=\left\{(x,0)\in\Bbbk^2\right\}$ we see that $J\left(\mathcal{V}_1\right)+J\left(\mathcal{V}_2\right)=\left\langle x^2,y\right\rangle$ is not radical, the functional $\lambda(g):\frac{\partial g}{\partial x}(0,0)$ belongs to $[J(\mathcal{V}_{1})+J(\mathcal{V}_{2})]^{\bot}$ hence the data $p_1,p_2$ must satisfy the two conditions: $\left(p_1-p_2\right)(0,0)=\frac{\partial\left(p_1-p_2\right)}{\partial x}(0,0)=0$. The second of these condition was not satisfied by the data considered in the example.
\end{remark}

Unfortunately, at this moment, we do not have a good condition on the data $\left(p_j\right)$ that would be equivalent to the solvability of an interpolation problem for more than two varieties. Here is a sufficient condition:

\begin{theorem}\label{thm2.8}
Let $J_i:=J\left(\mathcal{V}_i\right)$, $i=1,\dotsc,n$ be given and $p_1,\dotsc,p_n$ be the polynomial data on the varieties $\mathcal{V}_i$. Suppose that for any $i,j$
$$
\left(p_i-p_j\right)\in J_i+\bigcap_{k\ne i}J_k.
$$
Then the interpolation problem has a solution.
\end{theorem}

\begin{proof}
Let
$$
f_i:=p_i-\frac1n\left(\sum_{j=1}^np_j\right)
$$
then $f_i\in J_i+\bigcap\limits_{k\ne i}J_k$. Indeed $f_i=\frac1n\sum\limits_{j\ne i}^n\left(p_i-p_j\right)$ and each term in the sum belongs to $J_{i}+\bigcap\limits_{k\ne i}J_{j}$ hence the sum belongs to that ideal.

By Lemma \ref{lem2.6} this implies that there exists a polynomial $g_i$ such that $g_i-f_i\in J_i$ and $g_i\in\bigcap\limits_{k\ne i}J_k$.

Consider the polynomial%
$$
f:=\frac1n\left(\sum_{k=1}^np_k\right)+\sum_{i=1}^ng_i.
$$
Notice that each polynomial $g_{i}$ vanishes on $\mathcal{V}_{j}$ for $j\ne i$. Hence $f\mid\mathcal{V}_j=\left(\frac1n\left(\sum\limits_{k=1}^np_k\right)+g_j\right)
\mid\mathcal{V}_j=\left(\frac1n\left(\sum\limits_{k=1}^np_k\right)+f_j\right)\mid\mathcal{V}_j$, (since $f_j-p_j=0$ on $\mathcal{V}_j$) and this shows that $f\mid\mathcal{V}_j=\left(\frac1n\left(\sum\limits_{k=1}^np_k\right)+p_j-\frac1n\left(\sum\limits_{k=1}^np_k\right)\right)\mid\mathcal{V}_j=p_j\mid \mathcal{V}_j$.
\end{proof}

To show that this condition is not necessary consider again the three varieties in $\mathbb{C}^3$ that are defined by ideals $J_1=\langle z-x\rangle$, $J_2=\langle z-y\rangle$ and $J_3=\langle x,y\rangle$. The function $f(x,y,z)=z$ interpolates $p_1(x,y,z)=x$ on $\mathcal{V}\left(J_1\right)$, it interpolates $p_2(x,y,z)=y$ on $\mathcal{V}\left(J_2\right)$ and, of course, $p_3(x,y,z)=z$ on $\mathcal{V}\left(J_3\right)$. On the other hand $p_3-p_2=z-x$ does not belong to the ideal $\left(J_1\cap J_2\right)+J_3=\langle (z-x)(z-y)\rangle+\langle x,y\rangle=\left\langle z^2-zx-zy+xy,x,y\right\rangle=\left\langle x,y,z^2\right\rangle$.

\section{Interpolation and boundary value problems for PDEs}\label{sec3}

In this section we will consider interpolation on varieties not by the general polynomials but by polynomials that belong to a particular subspace of polynomials, namely subspace formed by the solutions of linear homogeneous differential equations. In particular we will extend result of W. K. Hayman and Z. G. Shanidze \cite{Hay:Sha} who considered this problem only for homogeneous quadratic differential equations in two variables.

For a polynomial $L\left(x_1,\dotsc,x_d\right)\in\mathbb{C}\left[\mathbf{x}\right]$ we use $L(\mathbf{D})$ to denote the differential operator obtained by formally replacing the variables $x_1,\dotsc,x_d$ in the polynomial $L$ with partial derivatives $D_1,\dotsc,D_d$.

We will need the following:

\begin{theorem}[Matsuura \cite{Mat}]\label{thm3.1}
Let $L$ be a polynomial in $d$ variables and the lowest degree of monomial in $L(\mathbf{x})$ is $l$. Let $\mathbb{C}_{\le n}[\mathbf{x}]$ be the polynomials of degree at most $n$. Then
$$
L(\mathbf{D})\mathbb{C}_{\le n}[\mathbf{x}]
=\mathbb{C}_{\le n-l}[\mathbf{x}]
$$
\end{theorem}

\begin{theorem}\label{thm3.2}
Let $J=\langle q\rangle$ be a radical ideal generated by a polynomial $q\in\mathbb{C}\left[\mathbf{x}\right]$ of degree $l$. Let $L$ be a polynomial in $d$ variables and the lowest degree of non-vanishing monomial in $L(\mathbf{x})$ is $l$. If
\begin{equation}
\ker L(\mathbf{D})\cap J=\{0\}
\label{eq3.1}
\end{equation}
then for any polynomial $p$ there exists unique polynomial $f\in\ker L(\mathbf{D})$ such that $f\mid\mathcal{V}(J)=p\mid\mathcal{V}(J)$.

Moreover this polynomial is a polynomial with minimal degree that interpolates $p$ on $\mathcal{V}(J)$.
\end{theorem}

\begin{proof}
Consider $L(\mathbf{D})$ as a mapping from $\mathbb{C}_{\le n}[\mathbf{x}]\to\mathbb{C}_{\le n}[\mathbf{x}]$. Then $\dim\left(\ker L(\mathbf{D})\cap
\mathbb{C}_{\le n}[\mathbf{x}]\right)+\dim\left(\operatorname{Im}L(\mathbf{D})\cap\mathbb{C}_{\le n}[\mathbf{x}]\right)=\dim\mathbb{C}_{\le n}[\mathbf{x}]$ and, by Theorem \ref{thm3.1}, $\operatorname{Im}L(\mathbf{D})\cap\mathbb{C}_{\le n}[\mathbf{x}]=\mathbb{C}_{\le n-l}[\mathbf{x}]$. Thus
$$
\dim\left(\ker L(\mathbf{D})\cap\mathbb{C}_{\le n}[\mathbf{x}]\right)
=\dim\mathbb{C}_{\le n}[\mathbf{x}]-\dim\mathbb{C}_{\le n-l}[\mathbf{x}].
$$
On the other hand, since $J$ is generated by polynomial $q$ of degree $l$, it follows that $J\cap\mathbb{C}_{\le n}[\mathbf{x}]=q\cdot\mathbb{C}_{\le n-l}[\mathbf{x}]$ and $\dim\left(J\cap\mathbb{C}_{\le n}[\mathbf{x}]\right)=\dim\mathbb{C}_{\le n-l}[\mathbf{x}]$. Thus
$$
\dim\left(\ker L(\mathbf{D})\cap\mathbb{C}_{\le n}[\mathbf{x}]\right)
+\dim\left(J\cap\mathbb{C}_{\le n}[\mathbf{x}]\right)
=\dim\mathbf{\mathbb{C}_{\le n}[\mathbf{x}]}
$$
and, since by assumption,
$$
\left(\ker L(\mathbf{D})\cap\mathbb{C}_{\le n}[\mathbf{x}]\right)\cap\left(J\cap\mathbb{C}_{\le n}[\mathbf{x}]\right)=\{0\}
$$
we conclude that
\begin{equation}
\mathbb{C}_{\le n}[\mathbf{x}]
=(\ker L(\mathbf{D})\cap\mathbb{C}_{\le n}[\mathbf{x}])\oplus(J\cap\mathbb{C}_{\le n}[\mathbf{x}]).
\label{eq3.2}
\end{equation}
Since this is true for all $n$, it follows that $\mathbb{C}\left[\mathbf{x}\right]=\ker L(\mathbf{D})\oplus J$ and thus for every $p\in\mathbb{C}\left[\mathbf{x}\right]$ there exists unique $f\in\ker L(\mathbf{D}))$ such that $(f-p)\in J$ and, in particular, $f\in\ker L(\mathbf{D}))$ interpolates $p$ on $\mathcal{V}(J)$.

For the ``moreover'' part of the theorem let $f_1$ be a polynomial of minimal degree that interpolates $p$ on $\mathcal{V}(J)$ and suppose that $\deg f_1=m$. Then by \eqref{eq3.2} there exists a polynomial $f\in\ker L(\mathbf{D)\cap\mathbb{C}_{\le m}[\mathbf{x}]}$ (thus of degree $m$) such that $f$ interpolates $f_1$, and hence, $p$ on $\mathcal{V}(J)$.
\end{proof}

Notice that in the language of PDEs the theorem says that under the assumption \eqref{eq3.1} the boundary value problem
$$
\left\{
\begin{array}{c}
L(\mathbf{D)}f=0 \\
f\mid\mathcal{V}=p
\end{array}
\right.
$$
has a polynomial solution and, in fact, such solution is unique.

\begin{remark}\hfill
\begin{enumerate}
\item[a)]  In Theorem \ref{thm3.2} we only dealt with an ideal $J$ generated by one polynomial $J=\langle q\rangle$. If the ideal $J$ is generated by several polynomial $J=\left\langle q_1,\dotsc,q_n\right\rangle$ we can apply the theorem to the ideal generated, by say, $J_1:=\left\langle q_1\right\rangle$. Then, since $\mathcal{V}(J)\subset\mathcal{V}\left(J_1\right)$ it follows that every $f$ that interpolates $p$ on $\mathcal{V}\left(J_1\right)$ also interpolates $p$ on $\mathcal{V}(J)$. Of course in this case we cannot guarantee uniqueness.

\item[b)]  Since the variety $\mathcal{V}$ was not assumed to be irreducible, the theorem also applies to the ``boundary value problem''
$$
\left\{
\begin{array}{c}
L(\mathbf{D)}f=0 \\
f\mid\mathcal{V}_j=p_j\text{, }j=1,\dotsc,n
\end{array}
\right..
$$

In this case we have to assume that there exists a polynomial $p$ that interpolates $p_j$ on $\mathcal{V}_j$ and apply the theorem to polynomial $p$ on $\mathcal{V}:=\bigcup\limits_{j=1}^n\mathcal{V}_j$.

\item[c)]  The condition $\ker L(\mathbf{D})\cap J=\{0\}$ is clearly necessary and sufficient for the uniqueness of interpolation, if it exists. On the other
hand, we saw that the condition guarantees the existence of an interpolation. This is a complete analogue of an elementary result in linear algebra: the equation $Ax=b$ for an $n\times n$ matrix $A$ has a solution for all $b$ if and only if the solution, if it exists, is unique.
\end{enumerate}
\end{remark}

Next, we will show that for any ideal $J=\langle q\rangle$ there exists a homogeneous polynomial $L$ such that $\ker\overline{L}(D)$ satisfies the conditions of
Theorem \ref{thm3.2}. In fact the polynomial $L=q^{\wedge}$ will do the job. (Here $\bar{L}$ denotes the complex conjugate polynomial to $L$.)

Here are a few preliminaries:

We let $\mathbb{H}_k[\mathbf{x}]$ denote the space of homogeneous polynomials in $\mathbf{x}$ of degree $k$. For a polynomial $L(\mathbf{x})=L\left(x_1,\dotsc,x_d\right)$ we let $L(\mathbf{D})=L\left(D_1,\dotsc,D_d\right)$ be the differential operator obtained by formally replacing the variables $x_j$ with the differential operators $D_{j}$. Next we introduce a Hermitian inner product on $\mathbb{H}_k[\mathbf{x}]$
defined by $\langle f,L\rangle=\overline{L}(D)f$. Observe that
$$
\left\langle\mathbf{x}^a,\mathbf{x}^{\beta}\right\rangle
=\left\{
\begin{array}{ccc}
0 & \text{if} & \alpha \neq \beta \\
\alpha ! & \text{if} & \alpha =\beta%
\end{array}
\right..
$$
Hence, it is easy to see that, if $f(\mathbf{x})=\sum\limits_{\left\vert\alpha\right\vert=k}a_a\mathbf{x}^a$ and $L(x)=\sum_{\left\vert\alpha\right\vert=k}b_a\mathbf{x}^a$ then $\langle f,L\rangle=\sum\limits_{\left\vert\alpha
\right\vert=k}\alpha !a_{\alpha}\overline{b}_a$ where $\alpha !=(\alpha_1,\dotsc,\alpha_{d})!=\prod\limits_{j=1}^d\alpha_j!$. It also follows from straightforward computations that for any $j=1,\dotsc,d$ and any $L\in\mathbb{H}_{k+1}[\mathbf{x}]$ and $f\in\mathbb{H}_k[\mathbf{x}]$ we have
\begin{equation}
\left\langle x_jf,L\right\rangle=\left\langle f,D_jL\right\rangle.
\label{eq3.3}
\end{equation}

Here the first inner product is in $\mathbb{H}_{k+1}[\mathbf{x}]$ while the second is in $\mathbb{H}_k[\mathbf{x}]$.

\begin{theorem}\label{thm3.4}
Let $J=\langle q\rangle$ and let $L=q^{\wedge}$. Then $\ker\overline{L}(\mathbf{D})\cap J=\{0\}$ and hence, by Theorem \ref{thm3.2}, $\mathbb{C}\left[\mathbf{x}\right]=\ker\overline{L}(\mathbf{D})\oplus J$.
\end{theorem}

\begin{proof}
Let $\mathbb{F}_k=\left\{F\in\mathbb{H}_{k}[\mathbf{x}]:<f,F>=0\text{, }\forall f\in J^{\wedge}\cap\mathbb{H}_k[\mathbf{x}]\right\}$. Hence $\mathbb{F}_k=\left(J^{\wedge}\cap\mathbb{H}_k[\mathbf{x}]\right)^{\bot}$ and
\begin{equation}
\mathbb{F}_k\oplus\left(J^{\wedge}\cap\mathbb{H}_k[\mathbf{x}]\right)
=\mathbb{H}_k[\mathbf{x}].
\label{eq3.4}
\end{equation}
Next we assume that $F\in\mathbb{F}_{k+1}$. Then $\left\langle F,x_jf\right\rangle=0$ for all $f\in J^{\wedge}\cap\mathbb{H}_k[\mathbf{x}]$ for all $j=1,\dotsc,d$ hence, by \eqref{eq3.3}, $\left\langle D_jF,f\right\rangle=0$ for $f\in f\in J^{\wedge}\cap\mathbb{H}_k[\mathbf{x}]$. In other words
$$
F\in\mathbb{F}_{k+1}\text{ implies }D_jF\in\mathbb{F}_k.
$$
Next we want to show that
\begin{equation}
\mathbb{F}_k\subset\ker\overline{L}(\mathbf{D})
\label{eq3.5}
\end{equation}
for every $k$. By definition of $\mathbb{F}_l$, for every $F\in\mathbb{F}_l$ we have $0=\left\langle q^{\wedge},F\right\rangle=\left\langle F,q^{\wedge}\right\rangle$ hence $\mathbb{F}_l\subset\ker\overline{L}$. By induction, assume that $\mathbb{F}_k\subset\ker\overline{L}$ and let $F\in\mathbb{F}_{k+1}$. Then $D_jF\in\mathbb{F}_{k+1}$, for every $j=1,\dotsc,d$. Hence $\overline{L}(\mathbf{D)(}D_jF)=0$ and thus $D_j(\overline{L}(\mathbf{D})F)=0$. But $\overline{L}(\mathbf{D)}F$ is a homogeneous polynomial and if all of its partial derivatives equal to zero. Hence $\overline{L}(\mathbf{D)}F=0$.
Finally observe that, since $\deg q^{\wedge}=l$, we have $\mathbb{C}_{<l}[\mathbf{x}]\subset\ker\overline{L}(\mathbf{D)}$ and our conclusion will follow from \eqref{eq3.4} and the fact that
$$
\mathbb{F}_k=\ker\overline{L}(\mathbf{D})\cap\mathbb{H}_k[\mathbf{x}].
$$
In view of \eqref{eq3.5} it is sufficient to show that
\begin{equation}
\dim\mathbb{F}_k
=\dim\left(\ker\overline{L}(\mathbf{D})\cap\mathbb{H}_k[\mathbf{x}]\right).
\label{eq3.6}
\end{equation}
Observe that, by \eqref{eq3.4} $\dim\mathbb{F}_k=\dim\mathbb{H}_k[\mathbf{x}]-\dim\left(J^{\wedge}\cap\mathbb{H}_k[\mathbf{x}]\right)$. Since $J\cap\mathbb{C}_{\le k}[\mathbf{x}]=q\cdot\mathbb{C}_{\le k-l}[\mathbf{x}]$, it follows that $\dim\left(J^{\wedge}\cap\mathbb{H}_k[\mathbf{x}]\right)=\dim\mathbb{H}_{k-l}[\mathbf{x}]$ and thus
$$
\dim\mathbb{F}_k
=\dim\mathbb{H}_k[\mathbf{x}]
-\dim\mathbb{H}_{k-l}[\mathbf{x}].
$$
On the other hand consider $\overline{L}(\mathbf{D})$ as an operator on $\mathbb{H}_k[\mathbf{x}]$. By Theorem \ref{thm3.1} its image is $\mathbb{H}_{k-l}[\mathbf{x}]$ hence
\begin{gather*}
\dim\left(\ker\overline{L}(\mathbf{D})\cap\mathbb{H}_k[\mathbf{x}]\right) \\
=\dim\mathbb{H}_k[\mathbf{x}]-\dim\left(\overline{L}(\mathbf{D})\mathbb{H}_k[\mathbf{x}]\right) \\
=\dim\mathbb{H}_k[\mathbf{x}]-\dim\mathbb{H}_{k-l}[\mathbf{x}] \\
=\dim\mathbb{F}_k.
\end{gather*}
This establishes \eqref{eq3.6} and thus proves the theorem.
\end{proof}

\begin{remark}
The last theorem has an interesting relation to a question asked in \cite{Cla:McK:She}: What subspaces of $\mathbb{C}\left[\mathbf{x}\right]$ have a complement which is an ideal in $\mathbb{C}\left[\mathbf{x}\right]$? The last theorem shows that for any homogeneous polynomial $L$ the space of polynomial solutions of the differential equation $\overline{L}(\mathbf{D})f=0$ has an ideal complement, namely any ideal generated by a polynomial $q$ with $q^{\wedge}=L$. In particular, the space of all harmonic polynomials in $d$ variables is complemented by the ideal $J=\left\langle x_1^2+\dotsb+x_d^2\right\rangle$.
\end{remark}

\section*{Acknowlegements}

We wish to thank the referee for many useful corrections and suggestions.

\bibliography{arbitrary}
\bibliographystyle{plain}

\end{document}